\documentclass[12pt]{article}

\usepackage{amsmath}    % need for subequations
\usepackage{amsthm} % for theorems and proofs environments
\usepackage{graphicx}
\usepackage{epic,eepic}

\theoremstyle{plain}% default
\newtheorem{thm}{Theorem}[section]
\newtheorem{lem}[thm]{Lemma}
\newtheorem{prop}[thm]{Proposition}
\newtheorem{cor}{Corollary}
\newtheorem{assumption}{Assumption}

\theoremstyle{definition}
\newtheorem{defn}{Definition}[section]

\newtheorem{exmp}{Example}[section]

\theoremstyle{remark}
\newtheorem{rem}{Remark}

\newcommand{\Ham}{\mathcal{H}}

\newcommand{\opt}[1]{\hat{#1}}
\newcommand{\var}[1]{#1_{\alpha}}
\newcommand{\dd}{{\rm\,d}}

\newcommand{\eps}{\varepsilon}

\newcommand{\abs}[1]{\left\vert#1\right\vert}

\begin{document}

\title{Necessary Conditions for Infinite Horizon Optimal Control Problems Revisited}
\author{Anton O. Belyakov\\ Lomonosov Moscow State University}

\maketitle

\begin{abstract}
Necessary optimality conditions in the form of the maximum principle for control problems with infinite time horizon are considered. Both finite and infinite values of objective functional are allowed since the concept of overtaking and weakly overtaking optimality is used. New form of optimality condition is obtained and compared with the transversality conditions usually used in the literature. The examples, where these transversality conditions may fail while the new condition holds are presented. For Ramsey problem of capital accumulation a simple form of necessary optimality conditions is derived, which is also valid in the case of zero discounting.
\end{abstract}

%\begin{keyword}
%\end{keyword}

\section{Introduction}\label{sec:intro}
Optimal control problems with infinite horizon play an important role in economic theory. For instance, in the theory of economic growth, 
%starting from celebrated Ramsey model, 
Pontragin's maximum principle is the workhorse for many researchers. The proof of the maximum principle for infinite time horizon one can find, e.g., in \cite{Halkin1974}. The proved theorem does not include transversality conditions. Moreover it is known, \cite{Halkin1974,Kamihigashi2001,Shell1969}, that usually used forms of transversality conditions:
\begin{equation}\label{eq:tcPSI}
  \lim_{t\rightarrow\infty} \psi(t) = 0,
\end{equation}
\begin{equation}\label{eq:tcXPSI}
  \lim_{t\rightarrow\infty} \langle\opt{x}(t),\psi(t)\rangle = 0,
\end{equation}
can be not necessary, where $\opt{x}$ is the optimal \emph{state variable}, $\psi$ is the corresponding \emph{adjoint variable}, and brackets $\langle \cdot,\cdot\rangle$ denote scalar product
of two vectors. %Section 5 in \cite{Halkin1974} provides example (iii), where condition (\ref{eq:tcPSI}) is violated for optimal control.
%Notice that in one-dimensional case, necessity of condition (\ref{eq:tcXPSI}) requires, can be not necessary, see e.g., .

Transversality condition obtained in \cite{Michel1982}, under  assumptions including that the objective functional takes only finite values, has the form of \emph{Hamiltonian} $\Ham$ converging to zero
\begin{equation}\label{eq:tcM}
  \lim_{t\rightarrow\infty} \Ham(\opt{x}(t),\opt{u}(t),t,\psi(t)) = 0,
\end{equation}
where $\opt{x}$ and $\opt{u}$ are the optimal state trajectory and control. \cite[Lecture III]{Shell1969} proposes, without prove, transversality condition as
\begin{equation}\label{eq:tcShell}
\limsup_{t\rightarrow\infty} \langle\psi(t),x(t)\rangle \geq \liminf_{t\rightarrow\infty} \langle\psi(t),\opt{x}(t)\rangle ,
\end{equation}
where $x$ is any admissible path of the state variable.
\cite[footnote 4 for Lecture III]{Shell1969} says that ``this conjecture is related to a conjecture made by Kenneth J. Arrow in private correspondence.'' Notice that Arrow's sufficiency theorem contains condition that follows from (\ref{eq:tcShell}), which one can expect for problems, where the maximum principle provides both necessary and sufficient conditions of optimality.

In \cite{AseevBesovKryazhimskii2012,AseevKryazhimskii2004,AseevVeliov2012,AseevVeliovNeedle} the authors determine the adjoint variable uniquely by a Cauchy-type formula, that solves adjoint equation with transversality conditions in the following form
\begin{equation}\label{eq:tcKAV}
  \lim_{t\rightarrow\infty} Y(t)\,\psi(t) = 0,
\end{equation}
where $Y(t)$ is the fundamental matrix of the state equation linearized about the optimal solution, see eg. \cite{Khlopin2013, Khlopin2015}.

Due to their limitations all aforementioned transversality conditions fail to select the optimal solution of Ramsey problem without discounting, \cite{Ramsey1928}, where we consider diverging objective functional. Condition (\ref{eq:tcM}) is proved only for converging functionals, but it can hold if we modify the objective improper integral subtracting from its integrand the constant, such that for optimal solution the integral converges, see e.g., \cite[Section 7]{Cass1965}. Condition (\ref{eq:tcShell}) fails because $x$ is any admissible path of the state variable, it could hold if $x$ have been requited to satisfy the maximum principle.  Condition (\ref{eq:tcKAV}) fails because it is proved for problems, where optimal trajectory is in the interior of the domain of admissible trajectories, which is not the case for Ramsey problem.%, where a slight overconsumption lead capital out of the positive domain in finite time. 
%This is probably why (\ref{eq:tcPSI}) and (\ref{eq:tcXPSI}) also fail in Ramsey problem without discounting.
 
Ramsey problem without discounting can be solved with the necessary conditions obtained in this paper. In contrast to  (\ref{eq:tcPSI})--(\ref{eq:tcKAV}) new conditions do not contain explicitly the adjoint variable. The paper in hand considers another example, where some of the conditions (\ref{eq:tcPSI})--(\ref{eq:tcKAV})  do not hold, while the new conditions are valid. The proved conditions include condition (\ref{eq:tcKAV}) as a special case and extend its domain of applicability. 

\section{Statement of the problem}
\label{sec:statement}

Let $X$ be a nonempty open
convex subset of $R^n$, $U$ be an arbitrary nonempty set in
$R^m$. Let us consider the following optimal
control problem:
\begin{align}
  & \int_{t_0}^{\infty}g(x(t),u(t),t)\dd t  \label{eq:J}
\rightarrow\max_{u},\\
  & \dot x(t) = f(x(t),u(t),t),\quad x(t_0) = x_0, \label{eq:dx}
\end{align}
where $u(t)\in U$ and exists state variable $x(t)\in X$ for all $t \in [t_0,+\infty)$.  We call such control $u(\cdot)$ and state variable $x(\cdot)$ trajectories \emph{admissible}. 
Functions $f$ and $g$ are differentiable w.r.t. their first argument, $x$, and together with these partial derivatives are defined and
locally bounded, measurable in $t$ for every $(x, u)\in X \times U$, and continuous in $(x, u)$
for almost every $t \in [0,\infty)$.

Improper integral in (\ref{eq:J}) might not converge for any candidate for optimal control $\opt{u}(\cdot)$, i.e. the limit
\begin{equation}\label{eq:Jinfty}
    \lim_{T\rightarrow\infty} J(\opt{u}(\cdot),x_0,t_0,T),
\end{equation}
might fail to exist, or might be infinite, where we introduce the finite time horizon functional:
\begin{equation}\label{eq:JT}
    J(u(\cdot),x_0,t_0,T) = \int_{t_0}^{T}g(x(t),u(t),t)\dd t,
\end{equation}
subject to state equation (\ref{eq:dx}). Thus functional $J$  may be unbounded or oscillating as $T\rightarrow\infty$. So we consider more general   definitions of optimality.
\begin{defn}\label{def:LOO} An admissible control $\opt{u}(\cdot)$ is 
\emph{overtaking optimal} (OO) if for every admissible control $u(\cdot)$ and every scalar $\eps > 0$ there exists time $T=T(\eps,u(\cdot)) > t_0$ 
such that for all $T' \geq T$
\begin{equation}\label{eq:LOO} 
  J(u(\cdot),x_0,t_0,T') - J(\opt{u}(\cdot),x_0,t_0,T') \leq \eps.
\end{equation}
\end{defn}
\begin{defn}\label{def:LWOO} An admissible control $\opt{u}(\cdot)$ is weakly
overtaking optimal (WOO) if for every admissible control $u(\cdot)$, scalar $\eps > 0$, and time $T > t_0$ 
one can find $T' = T'(\eps,T,u(\cdot)) \geq T$ such that 
\begin{equation}\label{eq:LWOO} 
  J(u(\cdot),x_0,t_0,T') - J(\opt{u}(\cdot),x_0,t_0,T') \leq \eps.
\end{equation}
\end{defn}

These two definitions imply that  for all admissible controls $u(\cdot)$ 
$$\limsup_{T\rightarrow\infty} \left(J(u(\cdot),x_0,t_0,T) - J(\opt{u}(\cdot),x_0,t_0,T)\right) \leq 0,$$
for OO $\opt{u}(\cdot)$ in Definition~\ref{def:LOO}, and
$$\liminf_{T\rightarrow\infty} \left(J(u(\cdot),x_0,t_0,T) - J(\opt{u}(\cdot),x_0,t_0,T)\right) \leq 0,$$
for WOO $\opt{u}(\cdot)$ in Definition~\ref{def:LWOO}. It is clear that if $\opt{u}(\cdot)$ is OO, then it is also WOO. When usual optimality holds, i.e. finite limit exists  in (\ref{eq:Jinfty}) and for all admissible controls $u(\cdot)$ 
$$\limsup_{T\rightarrow\infty} J(u(\cdot),x_0,t_0,T)  \leq \lim_{T\rightarrow\infty} J(\opt{u}(\cdot),x_0,t_0,T),$$
then $\opt{u}(\cdot)$ is also both OO and WOO.

There are many definitions of optimality in the literature, for example, corresponding uniform optimalities, when $T$ and $T'$ do not depend on $u(\cdot)$ in Definitions \ref{def:LOO} and \ref{def:LWOO}, see e.g., \cite{Khlopin2013}. These definitions are stronger and may lead to absence of corresponding optimal solutions. WOO seems to be one of the weakest concepts for which optimality conditions can be proved in the form of Pontryagin's maximum principle. 

\section{Optimality conditions}
\label{sec:conditions}

 With the use of the adjoint
variable $\psi$ we introduce \emph{Hamiltonian}
\begin{equation}\label{eq:H}
    \Ham(x,u,t,\psi,\lambda) = \lambda\, g(x,u,t) + \langle \psi, f(x,u,t)\rangle
\end{equation}
where brackets $\langle \cdot,\cdot\rangle$ denote scalar product
of two vectors. It is known, see \cite{Halkin1974,Pontryagin1961,Shell1969}, that there exist scalar $\lambda\geq 0$ and vector $\psi_0$, such that $(\lambda,\psi_0)\neq 0$ and the
maximum principle holds:
\begin{equation}\label{eq:maxH}
    \Ham(\opt{x}(t),\opt{u}(t),t,\psi(t),\lambda) = \max_{u \in U}\Ham(\opt{x}(t),u,t,\psi(t),\lambda),
\end{equation}
along with the adjoint equation:
\begin{equation}\label{eq:adjH}
    -\dot\psi(t) = \frac{\partial\Ham}{\partial
    x}(\opt{x}(t),\opt{u}(t),t,\psi(t),\lambda), \quad \psi(t_0) = \psi_0.
\end{equation}
In order to select such $\lambda$ and $\psi_0$ with the use of transversality condition we consider the linearization of system
(\ref{eq:dx})
\begin{equation}
  \dot y(t) = \left(\frac{\partial f}{\partial
    x}(\opt{x}(t),\opt{u}(t),t)\right)\,y(t).\label{eq:dy}
\end{equation}
Its solution, for the given initial condition $y(\tau)$, can be
written with the use of the \emph{state-transition matrix}
$K(t,\tau)$:
\begin{equation}\label{eq:yK}
    y(t) = K(t,\tau)\,y(\tau).
\end{equation}

The following assumption determines how well linearized system
(\ref{eq:dy}) should approximate the original
system, so that the proposed
optimality condition holds along with the maximum principle.
\begin{assumption}\label{a:lim0}
For almost all time instances $\tau \geq
t_0$ directional derivative $\langle J_x, \zeta\rangle$ of the functional at the optimal trajectory $\opt{x}$ uniformly bounds from below the finite difference approximation of the directional derivative:
\begin{align*}
   \lim_{\alpha\rightarrow 0} \liminf_{T\rightarrow\infty} \Bigg(& \frac{J(\opt{u}(\cdot),\opt{x}(\tau)+\alpha \zeta,\tau,T)-J(\opt{u}(\cdot),\opt{x}(\tau),\tau,T)}{\alpha} - \langle \opt{J}_x(\tau,T),\zeta\rangle\Bigg) \geq 0,
\end{align*}
only with such perturbations of the initial
conditions, $x(\tau)=\hat{x}(\tau)+\alpha \zeta$, that result in admissible trajectories, i.e. $x(t)\in X$ in $[\tau, \infty)$,
where we denote
\begin{equation}\label{eq:Jx}
\opt{J}_x(\tau,T) := \int\limits_{\tau}^{T}K^*(t,\tau)\,\frac{\partial g}{\partial
    x}(\opt{x}(t),\opt{u}(t),t) \dd t .
\end{equation}
\end{assumption}

\begin{rem}\label{r:lim0} If function $g$ does not depend on $x$,
then Assumption~\ref{a:lim0} trivially holds.
If functions $f$ and $g$ are linear w.r.t. $x$,
then Assumption~\ref{a:lim0} holds as equality for any
scalar $\alpha > 0$ and direction $\zeta$. \footnote{Due to inequality in Assumption~\ref{a:lim0}, 
it is also satisfied when $g$ is convex and $f$ is linear w.r.t. $x$.} 
\end{rem}

 The next assumption is needed for admissibility of optimal trajectory variations in all directions. 
 \begin{assumption}\label{a:x2}
There exists a number $\beta(\tau) > 0$
such that for all $x(\tau)\in X$ satisfying the inequality $\abs{x(\tau)-\opt{x}(\tau)} < \beta(\tau)$, the initial value
problem (\ref{eq:dx})
%$$\dot x(t) = f(x(t),u(t),t),\quad x(t_0) = x_0,$$ 
with $u = \opt{u}$ and the initial condition $x(t_0) = x(\tau)$ at $t_0=\tau$ has a
solution $x(t)\in X$ for all $t\geq\tau$.
\end{assumption}

\begin{prop}[Necessary optimality condition] 
\label{p:limU}
Let Assumptions \ref{a:lim0} and \ref{a:x2} 
be fulfilled for an admissible pair $(\hat{u}(\cdot),\hat{x}(\cdot))$ and the following limit finite valued
\begin{equation}\label{eq:MJx}% Is it possible to use the existence of  \limsup_{T\rightarrow\infty}\abs{\opt{J}_x(\tau,T)} instead? 
\limsup_{T\rightarrow \infty}\abs{\opt{J}_x(\tau,T)} < +\infty.
\end{equation}
1) If control $\opt{u}$ is WOO, then for all $\tau\in[t_0,\infty)$ and $u\in U$
\begin{align}\label{eq:limUwoo}
&\liminf_{T\rightarrow \infty}\left( \Ham(\opt{x}(\tau),u,\tau,\opt{J}_x(\tau,T),1) - \Ham(\opt{x}(\tau),\opt{u}(\tau),\tau,\opt{J}_x(\tau,T),1)\right) \leq 0,
\end{align}
2) If control $\opt{u}$ is OO, then for all $\tau\in[t_0,\infty)$ and $u\in U$
\begin{align}\label{eq:limUoo}
&\limsup_{T\rightarrow \infty}\left( \Ham(\opt{x}(\tau),u,\tau,\opt{J}_x(\tau,T),1) - \Ham(\opt{x}(\tau),\opt{u}(\tau),\tau,\opt{J}_x(\tau,T),1)\right) \leq 0.
\end{align}
\end{prop}
\begin{proof}See Appendix \ref{ap:limU}. \end{proof}
Example~\ref{ex:oscil} demonstrate the application of this new optimality condition, when $\opt{J}_x(\tau,T)$ oscillating in $T$. 

Under additional assumption, that $\opt{J}_x(\tau,T)$ converges as $T\rightarrow\infty$, referred to as condition of dominating discount  in general form, see \cite{AseevVeliovNeedle}, the following corollary proves that maximum principle holds in normal case ($\lambda = 1$) and provides an explicit expression for the adjoint variable, which is equivalent to transversality condition (\ref{eq:tcKAV}). 
\begin{cor}[Dominating discounting]\label{r:lim0}
If control $\opt{u}$ is WOO, Assumptions~\ref{a:lim0} and \ref{a:x2}  hold, and the following limit exists 
\begin{align} \opt{\psi}(\tau) & := 
\lim_{T\rightarrow\infty} \opt{J}_x(\tau,T)  = \int\limits_{\tau}^{\infty}K^*(t,\tau)\,\frac{\partial g}{\partial
    x}(\opt{x}(t),\opt{u}(t),t) \dd t, \label{eq:psi0}
\end{align}
see (\ref{eq:Jx}),
then $\opt{\psi}$ solves adjoint system
(\ref{eq:adjH}) in the normal case ($\lambda = 1$) and the maximum principle holds: 
\begin{equation}\label{eq:maxH0}
    \Ham(\opt{x}(t),\opt{u}(t),t,\opt{\psi}(t),1) \leq \Ham(\opt{x}(t),u,t,\opt{\psi}(t),1),\quad \text{for all } u \in U.
\end{equation}
\end{cor}%
\begin{proof}
Conditions (\ref{eq:limUwoo}) and (\ref{eq:limUoo}) take the form of maximum principle (\ref{eq:maxH0}). Differentiation w.r.t. $\tau$ of the integral expression for vector-function 
$\opt{\psi}$ in (\ref{eq:psi0}) shows that this is a solution of the adjoint system
(\ref{eq:adjH}) in the normal case. Alternatively, see Proposition~\ref{p:lima}  for $a_0 = 0$ and $\lambda = 1$.
\end{proof}
The following proposition, without Assumption~\ref{a:x2}, allows not only for normal case ($\lambda = 1$), but also for abnormal one, when $\lambda = 0$, see Example~\ref{ex:lima}.

\begin{prop}[Special transversality conditions] \label{p:lima}
If control $\opt{u}$ is WOO, Assumption~\ref{a:lim0} holds, and the following limit exists 
\begin{equation}\label{eq:lima}
\lim_{T\rightarrow \infty}K^*(T,t_0)\,\psi(T) = a_0,
\end{equation}
where $\psi(\cdot)$ is the solution of adjoint equation (\ref{eq:adjH}) such that maximum condition (\ref{eq:maxH}) is fulfilled, and the limit in (\ref{eq:psi0}) exists		, then $$\psi(\tau) = K^*(t_0,\tau)\,a_0 + \lambda\opt{\psi}(\tau).$$
\end{prop}
\begin{proof}See Appendix \ref{ap:lima}. \end{proof}
The similar expression was obtained in \cite[Section~6]{AseevBesovKryazhimskii2012} under certain assumptions ensuring the existence of the limit in (\ref{eq:psi0}). It was proved that vector $a_0$ has nonnegative components when the problem is autonomous and monotonic in the state variable, i.e. $\frac{\partial g}{\partial
    x}(x,u) > 0$, $\frac{\partial f}{\partial
    x}(x,u) > 0$ for all $(x,u)\in X\times U$ and for all optimal trajectories $\opt{x}$ there exist $\tau\geq t_0$ and vector $u_{\tau}$ such that $f(\opt{x}(t),u_{\tau}) > 0$.
In \cite{Khlopin2015} was noted the statement onf the following corollary.
\begin{cor}\label{c:lima} 
The limit in (\ref{eq:psi0}) exists if, and only if, (\ref{eq:tcKAV}) holds.
\end{cor}
\begin{proof} Taking $a_0=0$ in Proposition~\ref{p:lima}, we have
\begin{equation}\label{eq:lim0}
\lim_{T\rightarrow \infty}K^*(T,t_0)\,\opt{\psi}(T) = \lim_{T\rightarrow \infty}Y(T)\,\opt{\psi}(T) = 0,
\end{equation}
where the state transition matrix $K(T,\tau) = Y(T)\, Y(\tau)^{-1}$ is expressed via the non-degenerate fundamental matrix $Y$, such that $Y(t_0) = I$.
 \end{proof}
 Notice that, in contrast to $\psi(\cdot)$, the solution $\opt{\psi}(\cdot)$ of the adjoint equation is not necessarily a correct adjoint variable, such that maximum condition is fulfilled for an optimal control, since we do not require Assumption~\ref{a:x2} in Proposition~\ref{p:lima}. Assumption~\ref{a:x2} is indeed violated in some classical models such as  Ramsey model, see Example \ref{ex:Ramsey}.
 %, where only Assumption~\ref{a:lim0} holds and $\opt{\psi} \equiv \opt{J}_x \equiv 0$, because integrand $g$ does not depend on state variable $x$. 
 
 In order to find optimal control without Assumption~\ref{a:x2}, we can formulate similar necessary optimality conditions to (\ref{eq:limUwoo})--(\ref{eq:limUoo}) when state variable $x$ is one-dimensional ($n=1$) and state equation (\ref{eq:dx}) does not depend explicitly on time. Consider set $G\subset U\times X$ of all admissible pairs $(u(\cdot),x(\cdot))$ satisfying maximum condition (\ref{eq:maxH}) with adjoint equation (\ref{eq:adjH}). % and not violating at any time $\tau\in[t_0,\infty)$ any constraints imposed. 
 Then, in the following Proposition we take the set $\opt{U}(x) = \{u\,:\,(u,x)\in G \} \subset U$ instead of $U$. This means that we find \emph{synthesis of control} for each $\psi_0$ in (\ref{eq:adjH}) which results in an admissible control.
 
\begin{prop}[Special necessary optimality condition for autonomous one-dimensional state equation] 
\label{p:limU1}
Let one-dimensional state variable $x$ be governed by an autonomous equation, Assumption \ref{a:lim0}
 fulfilled for an admissible pair $(\hat{u}(\cdot),\hat{x}(\cdot))$, and  limit (\ref{eq:MJx}) finite valued.\\[5pt]
1) If control $\opt{u}$ is WOO, then for all $\tau\in[t_0,\infty)$ and $u\in \opt{U}(\opt{x}(\tau))$ 
% such that $(u,\opt{x}(\tau))\in G$
\begin{align}\label{eq:limUwoo1}
&\liminf_{T\rightarrow \infty}\left( \Ham(\opt{x}(\tau),u,\tau,\opt{J}_x(\tau,T),1) - \Ham(\opt{x}(\tau),\opt{u}(\tau),\tau,\opt{J}_x(\tau,T),1)\right) \leq 0,
\end{align}
2) If control $\opt{u}$ is OO, then for all $\tau\in[t_0,\infty)$ and $u\in \opt{U}(\opt{x}(\tau))$
%$u$ such that $(u,\opt{x}(\tau))\in G$
\begin{align}\label{eq:limUoo1}
&\limsup_{T\rightarrow \infty}\left( \Ham(\opt{x}(\tau),u,\tau,\opt{J}_x(\tau,T),1) - \Ham(\opt{x}(\tau),\opt{u}(\tau),\tau,\opt{J}_x(\tau,T),1)\right) \leq 0,
\end{align}
where $\opt{U}(\opt{x}(\tau))$ is the set of all admissible controls satisfying maximum principle (\ref{eq:maxH})--(\ref{eq:adjH}), taken at $x=\opt{x}(\tau)$.
\end{prop}
\begin{proof}See Appendix \ref{ap:limU}. \end{proof}
 
\begin{rem}\label{r:maxg} This case is equivalent to imposing mixed constraints which borders coincide with some admissible
pairs $(u(\cdot),x(\cdot))$ satisfying maximum principle (\ref{eq:maxH})--(\ref{eq:adjH}).  
\end{rem}

When integrand $g$ does not depend on state variable $x$,  like in Ramsey model in Example \ref{ex:Ramsey}, then we can formulate the following corollary.
\begin{cor} 
\label{c:limU1}
Let state variable $x$ be one-dimensional, the state equation autonomous, and integrand $g$ independent of state variable. 
If control $\opt{u}$ is WOO, then for almost all $\tau\in[t_0,\infty)$ the integrand of the objective functional is maximal on admissible control values:
\begin{equation}\label{eq:gmax}
    g(\cdot,\opt{u}(\tau),\tau) = \max_u\{g(\cdot,u,\tau)\,|\,(u,\opt{x}(\tau))\in G\}.
\end{equation}
\end{cor}
\begin{proof}Assumption~\ref{a:lim0} trivially holds and limit (\ref{eq:MJx}) is finite valued, due to $\opt{J}_x(\tau,T)\equiv 0$. Necessary condition (\ref{eq:gmax}) follows from Proposition~\ref{p:limU1}. 
\end{proof}

\section{Examples}
\begin{exmp}[centralized Ramsey model without discounting]\label{ex:Ramsey}
We maximize aggregated constant relative risk aversion utility
$$\int_0^{\infty}\frac{c(t)^{1-\theta}}{1-\theta}\dd t\rightarrow\max_{c \geq 0},\quad \text{s.t.}\quad \dot k(t) = k(t)^{\alpha} -\delta k(t) - c(t),\quad k(t) \geq 0,$$
where $k(0)=k_0 > 0$, $\theta\neq 1$, $\theta > 0$, and $\alpha\in(0,1)$. We write Hamiltonian in normal case ($\lambda = 1$): 
$$\Ham(k,c,t,\psi,1) = \frac{c^{1-\theta}}{1-\theta} + \psi \left(k^{\alpha} -\delta k - c\right),$$
since the abnormal case ($\lambda = 0$) would lead to $\psi = 0$ and thus impossible due to $(\lambda,\psi)\neq 0$.
Maximum condition $c(t)^{-\theta} = \psi(t)$ and 
the adjoint equation $- \dot{\psi}(t) = \left(\alpha k(t)^{\alpha - 1}-\delta\right) \psi(t)$ result in the Euler equation
$$\frac{\dot c(t)}{c(t)} = \frac{\alpha k(t)^{\alpha - 1}-\delta}{\theta}.$$
It follows from the Euler equation and the state equation, that any admissible pair $(k,c)$, that does not violate constraint $k(t)\geq 0$, converges either to steady state $(k_*,c_*)$, where $k_* = \left(\delta/\alpha\right)^{\frac{1}{\alpha-1}}$ and $c_* = \left(1-\alpha\right)\left(\delta/\alpha\right)^{\frac{\alpha}{\alpha-1}} > 0$, or to $(\delta^{\frac{1}{\alpha-1}},0)$, where $k_* < \delta^{\frac{1}{\alpha-1}}$.
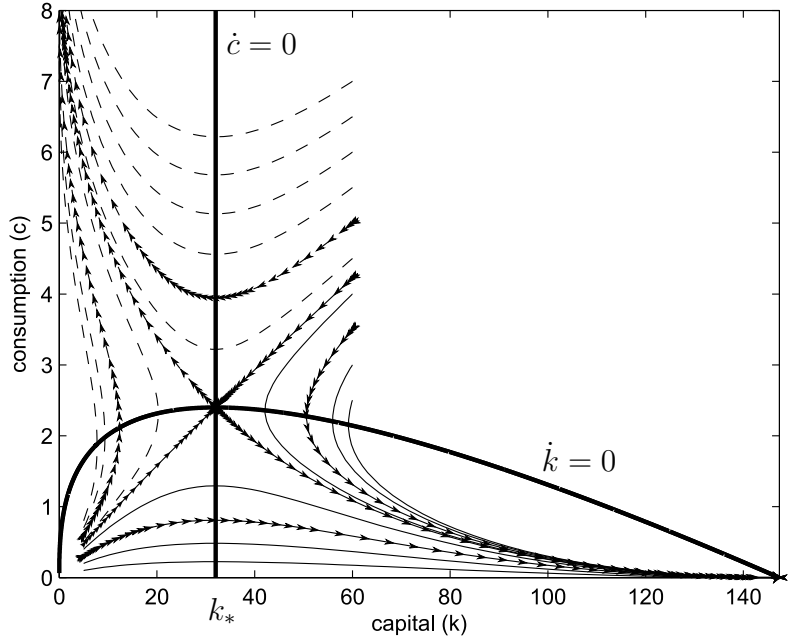
\begin{figure}[h]
\setlength{\unitlength}{0.009\textwidth}
\centering
\begin{picture}(100,80)
\put(0,0){\includegraphics[width=100\unitlength]{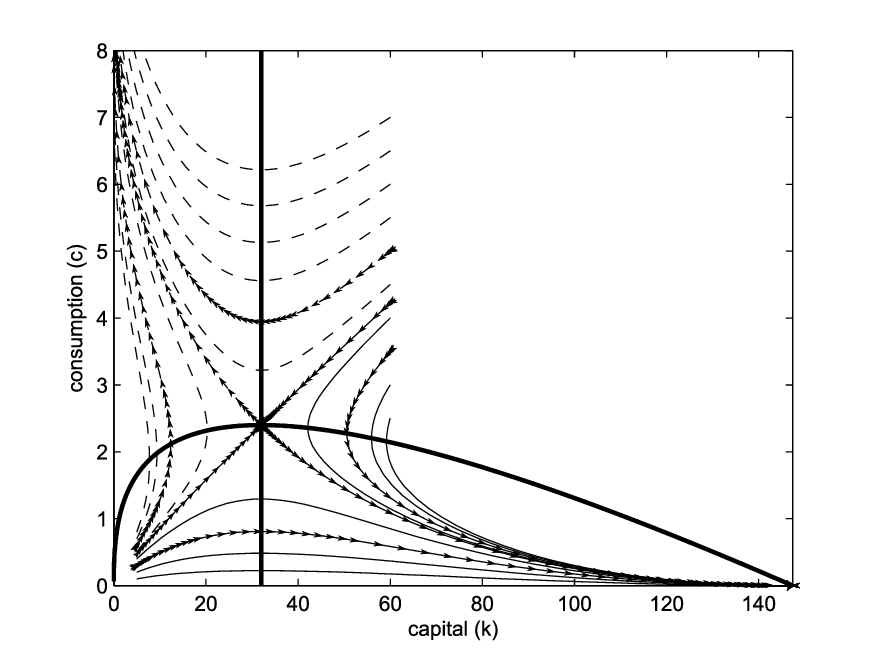}}
\put(29,4){$k_*$}
\put(31,65){$\dot c = 0$}
\put(65,20){$\dot k = 0$}
\end{picture}
\caption{Phase diagram for $\alpha = 0.4$, $\delta = 0.05$, and $\theta = 0.5$. Bold lines denote stationary points, where $\dot k =0$ and $\dot c = 0$. Solid lines are the trajectories satisfying optimality conditions in the form of Euler equation. Dashed lines are those trajectories which eventually violate nonegativity of capital $k$.}\label{f:Ramsey} 
\end{figure}
Necessary conditions (\ref{eq:gmax}) formulated in Corollary~\ref{c:limU1} %Remark~\ref{r:maxg} 
single out the optimal pair converging to $(k_*,c_*)$. 
\begin{equation}\label{eq:cmax}
    \frac{\opt{c}(t)^{1-\theta}}{1-\theta} \geq \frac{c^{1-\theta}}{1-\theta},\quad \text{ for all } c \text{ such that }  (c,k(t))\in G,
\end{equation}
where $G$ is the set of the trajectories governed by the state and Euler equations, not violating conditions $c(t) \geq 0$ and $k(t) > 0$, see solid lines in Figure~\ref{f:Ramsey}.
\end{exmp}

Let us check transversality conditions (\ref{eq:tcPSI})--(\ref{eq:tcKAV}):
\begin{enumerate}
\item Adjoint variable $\psi(t)=c(t)^{-\theta} \rightarrow c_{*}^{-\theta} > 0 $ does not converge to zero as $t\rightarrow\infty$.
\item The product $\langle\psi(t),\opt{k}(t)\rangle \rightarrow k_{*} c_{*}^{-\theta} > 0$ does not converge to zero. 
\item  Hamiltonian $\Ham(\opt{x}(t),\opt{u}(t),t,\psi(t),1) \rightarrow \frac{c_{*}^{1-\theta}}{1-\theta} >0$ does not converge to zero, because condition (\ref{eq:tcM}) requires convergence of the objective functional.
\item Each admissible pair of constants $(k,c)$, such that $k< k_*$ and $c = k^{\alpha} -\delta k$ would violate condition (\ref{eq:tcShell}). Notice that condition (\ref{eq:tcShell}) would be satisfied if we considered only such admissible trajectories that satisfy the maximum principle, i.e. the Euler equation.
\item Vector-function $\opt{J}_x$, obtained in (\ref{eq:exoscJx}), does not converge as $T\rightarrow\infty$, because condition (\ref{eq:tcKAV}) requires Assumption~\ref{a:x2}.
\end{enumerate}

The following example of an autonomous monotonic problem illustrates Proposition~\ref{p:lima}.
\begin{exmp}\label{ex:lima}
$$\int_0^{\infty}e^{-\rho t} x(t)\dd t\rightarrow\max_{u},\quad \text{s.t.}\quad \dot x(t) = u(t),\quad x(0)=0,\quad u(t)\in[0,1]$$
where obvious optimal control is $u\equiv 1$.

For $\rho > 0$ we have from Proposition~\ref{p:lima} the adjoint variable $\psi(t) = a_0 + \lambda e^{-\rho t}/\rho$ in normal case, where $\lambda > 0$. Overtaking optimal control $u\equiv 1$ provides maximum to the Hamiltonian
$\Ham(x,u,t,\psi,\lambda) = \lambda e^{-\rho t} x + \psi\, u$ for all $t\geq 0$ if, and only if, $a_0 \geq 0$. Thus, $\opt{\psi}(t) = \lambda e^{-\rho t}/\rho$ resulting from (\ref{eq:psi0}) is also valid. Notice that in the abnormal case of $\lambda = 0$ the adjoint is strictly positive, $\psi \equiv a_0 > 0$, since $(\lambda,\psi_0) \neq 0$.

For $\rho = 0$ maximum principle holds only in the abnormal form, $\psi \equiv a_0 > 0$, otherwise ($\lambda > 0$) maximum condition (\ref{eq:maxH}) would be violated for any solutions $\psi(t) = a_0 - \lambda t$. %Control is overtaking optimal.
In this case all  conditions (\ref{eq:tcPSI})--(\ref{eq:tcKAV}) are not satisfied.

All aforementioned correct adjoint solutions in this example, including the abnormal case, satisfy our new optimality condition (\ref{eq:limUoo}) that reads as $\psi_0 \geq 0$, since $K\equiv 1$ and $u \leq 1$, even though requirement (\ref{eq:MJx}) of Proposition~\ref{p:limU} is not fulfilled.

Let us check transversality conditions (\ref{eq:tcPSI})--(\ref{eq:tcKAV}) the case of $\rho = 0$, when maximum principle holds only with $\lambda = 0$:
\begin{enumerate}
\item Adjoint vector $\psi(t) \equiv a_0 > 0$ does not converge to zero as $t\rightarrow\infty$.
\item Due to (\ref{eq:oscilx}) the scalar product $\langle\psi(t),\opt{x}(t)\rangle = a_0 t$ does not converge to zero. 
\item Hamiltonian does not converge to zero.
\item Condition (\ref{eq:tcShell}) is not satisfied since  $\langle\psi(t),\opt{x}(t)-x(t)\rangle = -a_0 t < 0$ for $x(t)\equiv 0$. Notice that condition (\ref{eq:tcShell}) would be trivially satisfied, $\langle\psi(t),\opt{x}(t)-x(t)\rangle \equiv 0$ , if we considered only such admissible trajectories that satisfy the maximum principle, i.e. $x(t) = a_0 t$.
\item Vector-function $\opt{J}_x(\tau,T)=T-\tau$ does not converge as $T\rightarrow\infty$.
\end{enumerate}

\end{exmp}

\begin{exmp}[{\cite[Example 1.2]{Carlson1991infinite}}, if $b=0$]\label{ex:oscil}
Let us maximize the following integral
$$\max_u\int_0^{\infty} \left(x_2(t) + b u(t)\right)\dd t, $$
where $b \geq 0$, subject to the system describing a linear oscillator
\begin{align}
	\dot x_1(t) & = x_2(t),  & x_1(0) = 0,\\
	\dot x_2(t) & = u(t)  - x_1(t), & x_2(0) = 0,
\end{align}
with bounded control $u(t)\in[-1, 1]$. 
We have the state-transition matrix
$$K(t,\tau) = \left(\exp\left( 
\begin{array}{cc}
	0 & 1\\
	-1 & 0\\
\end{array}\right)\right)^{t-\tau} = \left( 
\begin{array}{cc}
	\cos(t-\tau) & \sin(t-\tau)\\
	-\sin(t-\tau) & \cos(t-\tau)\\
\end{array}\right)
$$
and the state trajectory
$$x(t) = \int\limits_{0}^{t} K(t,s) \left(\begin{array}{c}
	0\\
	u(s)\\
\end{array}\right) \dd s = \int\limits_{0}^{t} u(s)\left(\begin{array}{c}
	\sin(t-s)\\
	\cos(t-s)\\
\end{array}\right) \dd s. $$
Functional (\ref{eq:JT}) reads as follows
\begin{align}
J(u(\cdot),0,0,T) & = x_1(T) + \int_0^{T} b u(t) \dd t = \int_0^{T} \left(\sin(T-t) + b\right) u(t)\dd t.
\end{align}
For $b \geq 1$ control $\opt{u} \equiv 1$ is OO.
For $b\in [0,1)$ control $\opt{u} \equiv 1$ is WOO rather than OO, see proof in Appendix~\ref{ap:exlima}. This control maximizes the Hamiltonian 
\begin{equation}\label{eq:oscilH}
\Ham(x,u,t,\psi,\lambda) = \lambda \left(x_2 + b u\right) + \psi_1 x_2 + \psi_2 \left(u - x_1\right),
\end{equation}
only in normal case ($\lambda = 1$) and only for solutions 
\begin{equation}\label{eq:oscilpsi}
	\psi_1(t) = -r\cos(t+\phi) - 1,\quad
	\psi_2(t) = r\sin(t+\phi),
\end{equation}
of the adjoint system
\begin{align}
	\dot \psi_1(t) & = \psi_2(t),\\
	\dot \psi_2(t) & =  - \psi_1(t) - \lambda,
\end{align}
where $\abs{r}\leq b$ and $\phi$ is any phase shift. Indeed, in the abnormal case ($\lambda = 0$) 
control $\opt{u} \equiv 1$ would maximize Hamiltonian (\ref{eq:oscilH}) only for $\psi_1 \equiv \psi_2 \equiv 0$, that contradicted $(\lambda,\psi_0)\neq 0$.  Optimal trajectory reads as 
\begin{equation}\label{eq:oscilx}
	\opt{x}_1(t) = 1-\cos(t),\quad
	\opt{x}_2(t) =  \sin(t).
\end{equation}
 Vector-function $\opt{J}_x$, defined in (\ref{eq:Jx}), is oscillating in $T$:
\begin{align}
\opt{J}_x(\tau,T) = & \int\limits_{\tau}^{T} \left(\begin{array}{cc}
	\cos(t-\tau) & -\sin(t-\tau)\\
	\sin(t-\tau) & \cos(t-\tau)\\
\end{array}\right) \left(\begin{array}{c}
	0\\
	1\\
\end{array}\right) \dd t\nonumber\\ 
=  &\left(\begin{array}{c}
	\cos(T-\tau)-1\\
	\sin(T-\tau)\\
\end{array}\right).\label{eq:exoscJx}
\end{align}
We denote $\Delta\opt{\Ham}(u,\tau,T) := \Ham(\opt{x}(\tau),u,\tau,\opt{J}_x(\tau,T),1) - \Ham(\opt{x}(\tau),\opt{u}(\tau),\tau,\opt{J}_x(\tau,T),1)$
Sinse
$\Delta\opt{\Ham}(u,\tau,T) = \left(u-1\right) \left(\sin(T-\tau) + b\right)$,
 condition (\ref{eq:limUwoo}) for $\opt{u} \equiv 1$ being WOO takes the form $\left(u-1\right) \left(1 + b\right) \leq 0$ for all $u \in [-1,1]$. Condition (\ref{eq:limUoo}) for $\opt{u} \equiv 1$ being OO reads as $\left(u-1\right) \left(-1 + b\right) \leq 0$ for all $u \in [-1,1]$.
Let us check transversality conditions (\ref{eq:tcPSI})--(\ref{eq:tcKAV}):
\begin{enumerate}
\item Adjoint vector $\psi(t)$, obtained in (\ref{eq:oscilpsi}), does not converge to zero as $t\rightarrow\infty$.
\item Due to (\ref{eq:oscilx}) the scalar product 
$$\langle\psi(t),\opt{x}(t)\rangle = \cos(t) - r \cos(t+\phi) + r \cos(\phi) -1 \rightarrow 0$$ 
as $t\rightarrow\infty$ only for  
$\psi_1(t) = -\cos(t)-1$ and $\psi_2(t) = \sin(t)$ which belong to the correct adjoints in (\ref{eq:oscilpsi}) only if $b\geq 1$, when $\opt{u} \equiv 1$ is OO. 
\item  Due to (\ref{eq:oscilx}) and (\ref{eq:oscilpsi}) Hamiltonian $\Ham(\opt{x}(t),\opt{u}(t),t,\psi(t),1) =  r\sin(\phi) + b$
can converge to zero only if $r\sin(\phi) = -b$. Since $\abs{r}\leq b$ we have the only such adjoint solution:  $\psi_1(t) = -b\sin(t) - 1$ and $\psi_2(t) = -b\cos(t)$.
\item Condition (\ref{eq:tcShell}) seems to be satisfied for all correct adjoint variables (\ref{eq:oscilpsi}), such that $\abs{r}\leq b$.
\item Vector-function $\opt{J}_x$, obtained in (\ref{eq:exoscJx}), does not converge as $T\rightarrow\infty$.
\end{enumerate}

\end{exmp}

\appendix

\section{Proofs of Propositions \ref{p:limU} and \ref{p:limU1}}\label{ap:limU}
\begin{proof}[Proof of Proposition \ref{p:limU}]
\textbf{Needle variation} at
time $\tau$ can be defined as 
\begin{equation}
    u_{\alpha}(t) := \left\{\begin{array}{l@{,\quad}l}
                       \opt{u}(t) & t \notin (\tau-\alpha,\tau] \\
                       u & t \in (\tau-\alpha,\tau] \\
                     \end{array}\right.,
\end{equation}
where $u\in U$ is some constant. It is implied that $u =
u(\tau)$. Assumption~\ref{a:x2} guarantees that for sufficiently small $\alpha > 0$ control  $ u_{\alpha}$ is admissible, i.e. corresponding trajectory $\var{x}(t)\in X$ for all $t>t_0$.\\
\textbf{The corresponding increment in the value of the functional} can be written as follows:
\begin{eqnarray*}
  \Delta J_{\alpha}(T) &:=&	J(\var{u}(\cdot),x_0,t_0,T) - J(\opt{u}(\cdot),x_0,t_0,T)\nonumber\\
 &=&\int_{\tau - \alpha}^{T}\left(g(\var{x}(t),\opt{u}(t),t)-g(\opt{x}(t),\opt{u}(t),t)\right)\dd t \nonumber\\
&=&J(\opt{u}(\cdot),\var{x}(\tau),\tau,T) - J(\opt{u}(\cdot),\opt{x}(\tau),\tau,T) \nonumber\\
&&+\int_{\tau-\alpha}^{\tau}\left(g(\var{x}(t),u,t)-g(\opt{x}(t),\opt{u}(t),t)\right)\dd
t,
\end{eqnarray*}
where $\var{x}$ is the trajectory corresponding to control $\var{u}$.
Then, due to Assumption~\ref{a:lim0}, we have that for all $\eps > 0$ there exists $\alpha(\eps) > 0$ such that for all $T\geq\tau$ the following inequality holds
\begin{align}
  \frac{\Delta J_{\alpha}(T)}{\alpha}\! \geq & -\eps + \left\langle\opt{J}_x(\tau,T), \zeta \right\rangle \nonumber\\     
	& + \frac{1}{\alpha}\!\int_{\tau-\alpha}^{\tau}\left(g(\var{x}(t),u,t)-g(\opt{x}(t),\opt{u}(t),t)\right)\dd t.\label{eq:DJT}
\end{align}
%\textbf{Approximate solution:} 
We take the vector $\zeta$ in (\ref{eq:DJT}) as 
$$\var{\zeta}(\tau) = \frac{\var{x}(\tau) - \opt{x}(\tau)}{\alpha}.$$
Due to the differentiability of
function $f$ with respect to $x$ we have from differential equation (\ref{eq:dx}) the following limit
\begin{equation}\label{eq:xt}
\lim_{\alpha\rightarrow 0} \var{\zeta}(\tau) = y(\tau),
\end{equation}
where $y$ is the solution (\ref{eq:yK}) of the linearized system
(\ref{eq:dy}), $y(t) = K(t,\tau)\,y(\tau)$, with the following initial condition at
$t=\tau$:
\begin{equation}\label{eq:ytau}
y(\tau) = f(\opt{x}(\tau),u,\tau) -
f(\opt{x}(\tau),\opt{u}(\tau),\tau).
\end{equation} 
Taking into account expression (\ref{eq:xt}), we have the
following approximation of the last term in (\ref{eq:DJT}), due to continuity of $u$ and $\opt{u}$ at $\tau$ and continuity of $g$ w.r.t. $(x, u)$:
\begin{align*}
&\frac{1}{\alpha}\!\int\limits_{\tau-\alpha}^{\tau}\left(g(\var{x}(t),u,t)-g(\opt{x}(t),\opt{u}(t),t)\right)\dd
t \\ & =  g(\opt{x}(\tau),u,\tau)-g(\opt{x}(\tau),\opt{u}(\tau),\tau) +
O(\alpha),
\end{align*}
where $\lim_{\alpha\rightarrow 0} O(\alpha) = 0$. Hence, inequality (\ref{eq:DJT}) takes the form
\begin{align*}
  \frac{\Delta J_{\alpha}(T)}{\alpha} \geq & - \eps+ \left\langle\opt{J}_x(\tau,T), \var{\zeta}(\tau)\right\rangle
  + g(\opt{x}(\tau),u,\tau)-g(\opt{x}(\tau),\opt{u}(\tau),\tau) + O(\alpha).
\end{align*}
Limits (\ref{eq:MJx}) and (\ref{eq:xt}) imply %the inequality for the limit uniform in $T\in [\tau,\infty)$
that for all $\eps > 0$ there exist $T_1(\eps)$ and $\alpha_1(\eps) > 0$ such that for all $T\geq T_1(\eps)$ and $\alpha\in(0,\alpha_1(\eps)]$ holds
$$\left\langle
\opt{J}_x(\tau,T), y(\tau)-\var{\zeta}(\tau)\right\rangle \leq \abs{\opt{J}_x(\tau,T)}\abs{y(\tau)-\var{\zeta}(\tau)} < \eps,$$
that can be written as
$$\left\langle\opt{J}_x(\tau,T), y(\tau)\right\rangle < \left\langle\opt{J}_x(\tau,T), \var{\zeta}(\tau)\right\rangle + \eps.$$
Hence, we have inequality (\ref{eq:DJT}) in the form
\begin{align*}
  \frac{\Delta J_{\alpha}(T)}{\alpha} \geq & - 2\eps+ \left\langle\opt{J}_x(\tau,T), y(\tau)\right\rangle\\
  &+ g(\opt{x}(\tau),u(\tau),\tau)-g(\opt{x}(\tau),\opt{u}(\tau),\tau) + O(\alpha).
\end{align*}
\textbf{Definition \ref{def:LWOO} of WOO} means, that for all $\eps_2>0$ and $T_2>t_0$ there exists $T'(\eps_2)\geq T_2$ such that holds $\Delta J_{\alpha}(T') \leq \eps_2$.
Let us take $T_2 \geq T_1$ and $\eps_2 = \alpha\,\eps$. Then inequality $\Delta J_{\alpha}(T') \leq \alpha\,\eps$ results in
\begin{align}\label{eq:3eps}
  3\eps \geq &+ \left\langle\opt{J}_x(\tau,T'), y(\tau)\right\rangle
  + g(\opt{x}(\tau),u,\tau)-g(\opt{x}(\tau),\opt{u}(\tau),\tau) + O(\alpha).
\end{align}
Suppose that (\ref{eq:limUwoo}) is violated, i.e. there exist $\eps > 0$ and $T\geq t_0$ such that for all $T'\geq T$
\begin{align*}
  & \left\langle \opt{J}_x(\tau,T'), y(\tau)\right\rangle 
 + g(\opt{x}(\tau),u,\tau)-g(\opt{x}(\tau),\opt{u}(\tau),\tau) \geq 4\eps,
\end{align*}
Then we have contradiction with (\ref{eq:3eps}) taking $\alpha$ small enough, i.e (\ref{eq:limUwoo}) should hold. 

Similar calculations can be done for OO condition (\ref{eq:limUoo}).
\end{proof}

\begin{proof}[Proof of Proposition \ref{p:limU1}]
If set $\opt{U}(\opt{x}(\tau))$ is a singleton, then $\opt{u}(\tau)= \opt{U}(\opt{x}(\tau))$.
If set $\opt{U}(\opt{x}(\tau))$ contains many values we can construct an admissible 
\textbf{needle variation} at
time $\tau$ as 
\begin{equation}
    u_{\alpha}(t) := \left\{\begin{array}{l@{,\quad}l}
                       \opt{u}(t) & t \notin (\tau-\alpha,\tau] \\
                       u(t) & t \in (\tau-\alpha,\tau] \\
                     \end{array}\right.,
\end{equation} 
where $u(t)\in \opt{U}(\opt{x}(t))$ for all $t \in (\tau-\alpha,\tau]$. Admissibility of control $u_{\alpha}$ follows from convexity of $X$. Indeed, one-dimensional autonomous state equation will have an admissible solution for any subsequent combination of admissible controls. Thus for all $\alpha > 0$ corresponding trajectory $\var{x}(t)\in X$ for all $t>t_0$.

The rest of the proof is the same as that of Proposition \ref{p:limU}
\end{proof}

\section{Proof of Proposition \ref{p:lima}}\label{ap:lima}
First, we prove the following  Lemma. 
\begin{lem}\label{l:psiT}
Adjoint equation (\ref{eq:adjH}) can be written with the use of (\ref{eq:Jx}) as
\begin{align}\label{eq:psiT}
  \psi(\tau) & = K^*(T,\tau)\,\psi(T) + \lambda\,\opt{J}_x(\tau,T), \quad \psi(t_0) = \psi_0.
\end{align}
\end{lem}
\begin{proof} Since vector $y(\tau)$ in (\ref{eq:yK}) is arbitrary, from (\ref{eq:dy}) one can find the matrix derivative $$\frac{\partial K}{\partial t}(t,\tau) = \left(\frac{\partial f}{\partial x}(\opt{x}(t),\opt{u}(t),t)\right) K(t,\tau).$$ Taking the Hermitian transpose we have $$\frac{\partial K^*}{\partial t}(t,\tau) = K^*(t,\tau)\left(\frac{\partial f}{\partial x}(\opt{x}(t),\opt{u}(t),t)\right)^*.$$ Hence, if we multiply the adjoint equation (\ref{eq:adjH}) by matrix $K^*(t,\tau)$ as 
$$-K^*(t,\tau)\,\dot\psi(t) = K^*(t,\tau)\left(\frac{\partial f}{\partial x}(\opt{x}(t),\opt{u}(t),t)\right)^*\psi(t) + \lambda\,K^*(t,\tau)\,\frac{\partial g}{\partial x}(\opt{x}(t),\opt{u}(t),t),$$ 
then we have 
$$-\frac{\partial}{\partial t}\left(K^*(t,\tau)\,\psi(t)\right) = \lambda\, K^*(t,\tau)\,\frac{\partial g}{\partial x}(\opt{x}(t),\opt{u}(t),t).$$ Integration of the letter equation from $\tau$ till $T$ yields (\ref{eq:psiT}).
\end{proof}
\subsection{Main proof}
\begin{proof} 1) It follows from (\ref{eq:lima}) and (\ref{eq:psiT}) with $\lambda=1$  that
\begin{align*}
 \psi(\tau) - \lim_{T\rightarrow +\infty} \opt{J}_x(\tau,T) &  =  \lim_{T\rightarrow +\infty} K^*(T,\tau)\,\psi(T)\\
   &  = K^*(t_0,\tau) \lim_{T\rightarrow +\infty} \left(K^*(T,t_0)\,\psi(T)\right)\\
   &  = K^*(t_0,\tau)\, a_0,
\end{align*}
where we use the expression  $ K^*(T,\tau) = \left( K(T,t_0)\, K(t_0,\tau) \right)^* =   K^*(t_0,\tau)\, K^*(T,t_0)$.
Taking into account (\ref{eq:psi0}) we have $\psi(\tau)-\opt{\psi}(\tau) = K^*(t_0,\tau)\,a_0$.\\
 2) It follows from (\ref{eq:psiT}) with $\lambda = 0$, that  $\psi(\tau) = K^*(t_0,\tau)\,a_0$.
\end{proof}

\section{Proof of optimality in example~\ref{ex:lima}}\label{ap:exlima}
We show that for $b \in [0,1)$ control $u(t)\equiv 1$ is weak overtaking optimal among all controls in $[-1,1]$, see Definition~\ref{def:LWOO}, where 
\begin{align}
J(u(\cdot),0,0,T)-J(1,0,0,T) & = \int\limits_0^{T} \left(\sin(T-t) + b\right) \left(u(t)-1\right)\dd t.
\end{align}
It suffice to prove that for the function
\begin{align}
\Delta x_1(T)  := \int\limits_0^{T} \sin(T-t) \left(u(t)-1\right)\dd t
\end{align}
for all $T^{\prime}\geq 0$ there exists $T \geq T^{\prime}$ such that $\Delta x_1(T)\leq 0$. 
Assume the opposite, i.e. there exists $T^{\prime}\geq 0$ such that for all $T \geq T^{\prime}$ holds $\Delta x_1(T) > 0$.
Take the integer number $n$ such that $2\pi \left(n-1\right) > T^{\prime}$, so that $$\Delta x_1(2\pi \left(n-1\right)) = \int\limits_0^{2\pi \left(n-1\right)} \sin(t) \left(u(t)-1\right)\dd t > 0.$$
The function $\Delta x_1$ of $T=2n\pi$ can be written as
\begin{align}
\Delta x_1(2n\pi) =	&	-\int\limits_{0}^{2n\pi} \sin(t)\left(u(t)-1\right) \dd t\nonumber\\ 
= & -\Delta x_1(2\pi \left(n-1\right)) 
-\int\limits_{\left(2n-1\right)\pi}^{2n\pi} \sin(t)\left(u(t)-1\right) \dd t .
\end{align}
The last integral is not negative
$$\int\limits_{\left(2n-1\right)\pi}^{2n\pi}\sin(t)\left(u(t)-1\right) \dd t \geq 0,$$
since $u(t)\leq 1$ and $\sin(t) \leq 0$ for all $t\in [\left(2n-1\right)\pi, 2n\pi]$. Thus, recalling that $\Delta x_1(2\pi \left(n-1\right)) >0$, we have inequality $\Delta x_1(2n\pi) < 0$, which contradicts the assumption.\qed

\bibliographystyle{plain}%\bibliographystyle{apacite}
%\bibliography{biblio}

%--------------------------------------------------------------------------------------------------------------

\end{document}